\documentclass[10pt]{article}
\usepackage{amssymb,latexsym,amsmath,epsfig,amsthm}

\makeatletter

\renewcommand\section{\@startsection {section}{1}{\z@}
{-30pt \@plus -1ex \@minus -.2ex}
{2.3ex \@plus.2ex}
{\normalfont\normalsize\bfseries}}

\renewcommand\subsection{\@startsection{subsection}{2}{\z@}
{-3.25ex\@plus -1ex \@minus -.2ex}
{1.5ex \@plus .2ex}
{\normalfont\normalsize\bfseries}}

\renewcommand{\@seccntformat}[1]{\csname the#1\endcsname. }

\makeatother

\newtheorem{theorem}{Theorem}
\newtheorem{lemma}{Lemma}

\newtheorem{corollary}{Corollary}

\newtheorem{italicdefinition}{Definition}

\begin{document}

\begin{center}
\uppercase{\bf Equality Classes of Nim Positions under Mis\`{e}re Play}
\vskip 20pt
{\bf Mark Spindler}\\
{\tt mark.edward.spindler@gmail.com}
\end{center}
\vskip 30pt

\centerline{\bf Abstract}

\noindent
We determine the mis\`{e}re equivalence classes of \textsc{Nim} positions
under two equivalence relations: one based on playing disjunctive
sums with other impartial games, and one allowing sums with partizan
games.

In the impartial context, the only identifications we can make are
those stemming from the known fact about adding a heap of size 1.
In the partizan context, distinct \textsc{Nim} positions are inequivalent.

\pagestyle{myheadings} 
\thispagestyle{empty} 
\baselineskip=12.875pt 
\vskip 30pt

\section{Introduction}

In the case of normal play, there are three different important equivalence
relations one may impose on a class of games including the \textsc{Nim} positions.
One may wish to know enough about a \textsc{Nim} position to play disjunctive
sums with other \textsc{Nim} positions, or with arbitrary impartial games,
or with arbitrary partizan games. However, in all three cases, the
required information is the same: the Grundy number, obtained by taking
the bitwise xor of the heap sizes.

In mis\`{e}re play, the Grundy number is almost enough to play a
sum of Nim heaps (two extra equivalence classes are needed to properly
handle sums of heaps of size one). However, this is not nearly enough
information to play with arbitrary impartial games, let alone partizan
ones. 

If we restrict the context to only impartial games, then it is widely
known that adding a heap of size $1$ is the same, up to equivalence, as changing the $2^0$ bit in the binary representation of an existing Nim heap. Plambeck noted in subsection 11.1 of
\cite{Plambeck} that distinct sums of
size-$2$ heaps are inequivalent. We show that the \emph{only}
possible identifications are those related to adding Nim heaps
of size $1$. This makes formal Siegel's assertion about simplifications of mis\`{e}re \textsc{Nim} in \cite{siegelminor}.

Furthermore, as noted in \cite{siegel}, a sum
of two heaps of size $1$, which is equivalent to $0$ in the context
of impartial games, becomes inequivalent to $0$ if we allow partizan games. In
this paper, we show that in the context of partizan games, the \textsc{Nim}
positions are all pairwise inequivalent. 

In summary, this paper confirms that in these overly broad contexts (or ``universes'' in the language of \cite{ends}),
essentially all of the information about a \textsc{Nim} position is required
to play it in an arbitrary disjunctive sum.

\section{General Notation}

Until Section 6, all games in this paper are assumed to be short. That is, they have finitely many distinct subpositions, and admit no infinite runs.

In this paper, almost every specific position we refer to will be a \textsc{Nim} position. As such, numerals in standard typeface (e.g. $3$) will always denote sizes of Nim heaps. As is usual, we will be using $+$ to denote disjunctive sums, so that, for example, $3+5$ is not equivalent to $8$ in any way under discussion.

Two (possibly partizan) games are said to be \textit{isomorphic} if their
game trees are isomorphic. For example, the Kayles position with two adjacent
pins is isomorphic to the \textsc{Nim} position $2$. The game $2+3+0$ is isomorphic
to $3+2$, but not to $2+2+1$. If $G$ is isomorphic to $H$, we
write $G\cong H$. 

As we would like to compare sums of games under mis\`{e}re play, we adopt notation from \cite{cgt}.

We use $o^{-}\left(G\right)$ to denote the mis\`{e}re outcome
of $G$, which can be either $\mathcal{L}$, $\mathcal{R}$, $\mathcal{P}$,
or $\mathcal{N}$, according to whether Left, Right, the Previous, or the Next player has a winning strategy. Then there is a natural partial order on outcomes, with $\mathcal{L}\ge\mathcal{P}\ge\mathcal{R}$, $\mathcal{L}\ge\mathcal{N}\ge\mathcal{R}$, and $\mathcal{P}$ incomparable to $\mathcal{N}$. We can use this to
order games:

\[
G\ge H\text{ if and only if }o^{-}\left(G+X\right)\ge o^{-}\left(H+X\right)\text{ for all games }X\text{.}
\]
Then we can define equality (partizan equivalence) for games: 
\[
G=H\text{ if and only if }G\ge H\text{ and }G\le H
\]
Equivalently, 

\[
G=H\text{ if and only if }o^{-}\left(G+X\right)=o^{-}\left(H+X\right)\text{ for all games }X\text{.}
\]
There is a corresponding equivalence relation, \textit{impartial equivalence},
for the impartial context:
\[
G\equiv H\text{ if and only if }o^{-}\left(G+X\right)=o^{-}\left(H+X\right)\text{ for all impartial games }X\text{.}
\]

\section{Simplification of Games}

\subsection{Impartial Context}

In Chapter V of \cite{cgt}, there are some theorems which yield a recursive test for impartial equivalence 
Using $G'$ to denote an arbitrary option of an impartial game $G$:
\begin{defn}
A game $G$ is said to be \textit{linked} to $H$ if \[o^{-}\left(G+T\right)=o^{-}\left(H+T\right)=\mathcal{P}\text{ for some impartial }T\text{.}\]

\end{defn}
\renewcommand{\labelenumi}{\roman{enumi}.}

Theorem V.3.6 of \cite{cgt} states:
\begin{lemma}
Given impartial games $G$ and $H$, $G\equiv H$ if and only if the following conditions hold\end{lemma}
\begin{enumerate}
\item[(i)] $G$ is linked to no $H'$.
\item[(ii)] $H$ is linked to no $G'$.
\item[(iii)] If $G\cong0$, then $o^{-}\left(H\right)=\mathcal{N}$ and vice versa.
\end{enumerate}
And Theorem V.3.5 of \cite{cgt} states: 
\begin{lemma}
Given impartial games $G$ and $H$, $G$ is linked to $H$ if and only if no option of $G$ is impartially equivalent to $H$ and no option of $H$ is is impartially equivalent to $G$.
\end{lemma}
Combining these two theorems yields the following:
\begin{theorem}
\label{thm:impartial test}Given impartial games $G$ and $H$, $G\equiv H$ if and only if the following conditions hold:  \end{theorem}
\begin{enumerate}
\item[(i)] For every $H'$, there is either an option $H''$ with $H''\equiv G$ or an option $G'$ with $G'\equiv H'$.
\item[(ii)] For every $G'$, there is either an option $G''$ with $G''\equiv H$ or an option $H'$ with $H'\equiv G'$.
\item[(iii)] If $G\cong0$, $o^{-}\left(H\right)=\mathcal{N}$ and vice versa.
\end{enumerate}

\subsection{Partizan Context}

In Section V.6 of \cite{cgt}
, there are partizan versions of the theorems above. By restricting them to the case
in which the games are impartial, we will obtain a theorem for partizan
equivalence analogous to Theorem \ref{thm:impartial test}. 
\begin{defn}
A game $G$ is said to be \textit{downlinked} to $H$ if, for some $T$,

$o^{-}\left(G+T\right)\le\mathcal{P}$ and $o^{-}\left(H+T\right)\ge\mathcal{P}$.
\end{defn}
After a slight rewording, Theorem V.6.16 of \cite{cgt} states:
\begin{lemma}
$G\ge H$ if and only if the following conditions hold:\end{lemma}
\begin{enumerate}
\item[(i)] $G$ is downlinked to no $H^{L}$;
\item[(ii)] No $G^{R}$ is downlinked to $H$;
\item[(iii)] If $H$ has no Left options, then neither does $G$;
\item[(iv)] If $G$ has no Right options, then neither does $H$.
\end{enumerate}
Theorem V.6.15 of \cite{cgt} states: 
\begin{lemma}
$G$ is downlinked to $H$ if and only if no $G^{L}\ge H$ and $G\ge$no $H^{R}$.
\end{lemma}
Combining this with the previous theorem, we have a general recursive
test for $\ge$ as follows:
\begin{theorem}
$G\ge H$ if and only if the following conditions hold:\end{theorem}
\begin{enumerate}
\item[(i)] For every $H^{L}$, there is either an option $H^{LR}$ with $H^{LR}\le G$ or an option $G^{L}$ with $G^{L}\ge H^{L}$.
\item[(ii)] For every $G^{R}$, there is either an option $G^{RL}$ with $G^{RL}\ge H$ or an option $H^{R}$ with $H^{R}\le G^{R}$.
\item[(iii)] If $H$ has no Left options, then neither does $G$.
\item[(iv)] If $G$ has no Right options, then neither does $H$.
\end{enumerate}
If we restrict the previous theorem to the case in which $G$ and
$H$ are impartial, it simplifies considerably:
\begin{theorem}
\label{thm:partizantest}If $G$ and $H$ are impartial, then $G=H$
if and only if the following conditions hold:\end{theorem}
\begin{enumerate}
\item[(i)] For every $H'$, there is either an option $H''$ with $H''= G$ or an option $G'$ with $G'= H'$.
\item[(ii)] For every $G'$, there is either an option $G''$ with $G''= H$ or an option $H'$ with $H'= G'$.
\item[(iii)] If $G\cong0$, $H\cong0$ and vice versa.
\end{enumerate}
\begin{proof}
In this case, all of the inequalities reduce to $=$ and 
having no Left/Right options means a game is isomorphic to $0$. 
\end{proof}

\section{Partizan Equivalence Classes}

\subsection{Organizing the Nim Positions}

Every \textsc{Nim} position is a disjunctive sum of Nim heaps of various sizes.
Note that, for all games $G$, we have $G+0\cong G$. Also, permuting
the order of a sum of Nim heaps yields an isomorphic game. Therefore,
every \textsc{Nim} position is determined up to isomorphism by the multiset of nonzero heap sizes.
As such, we can identify a \textsc{Nim} position (up to isomorphism) with a
finite nondecreasing sequence of positive integers: $4+1+0+1$ is
identified with $\left(1,1,4\right)$, and $0+0+0$ is identified
with the sequence of length $0$, written $\left(\cdot\right)$ for
clarity.
\begin{defn}
Given two finite sequences of natural numbers $A$ and $B$, we say $A$ \textit{precedes} $B$ \textit{in quasi-lexicographic order}, and write $A\prec B$ if either $A$ is shorter than $B$, or they have the same length and $A$ precedes $B$ lexicographically (see example 5.1 in \cite{logic}). For example, $\left(2,9\right)\prec\left(2,2,2,4,5\right)\prec\left(2,2,4,4,5\right)$.
This is sometimes called \emph{radix} or \emph{shortlex} order. 
\end{defn}
Note that quasi-lexicographic order is a well-order on the set of all finite sequences. We will use this fact as the basis for several induction proofs to follow.
\begin{defn}
\label{orderdef}Given \textsc{Nim} positions $G$ and $H$, we say $G$ \textit{precedes} $H$, and write $G\prec H$, if the corresponding sequence
for $G$ precedes the one for $H$ in quasi-lexicographic order. \end{defn}
\begin{lemma}
\label{lem:options prec}If $G'$ is an option of a \textsc{Nim} position $G$, then $G'\prec G$.\end{lemma}
\begin{proof}
Let $G'$ be an arbitrary option of $G$. If the move to $G'$ involved
removing an entire heap, then $G'\prec G$ simply because the sequence
of heap sizes became shorter. Now suppose instead that the original
position is $G=\left(a_{1},\ldots,a_{N}\right)$, and the $j$th
heap has some, but not all, stones removed, with $b>0$ stones left
in that heap in $G'$. When the heap sizes of $G'$ are put in increasing
order, all the heap sizes of $G'$ agree with those of $G$ until
we reach the last heap with size equal to $b$. In that position,
$G$ necessarily has a higher heap size since $G'$ has more heaps
of size $b$ than $G$ does. Thus, by lexicographic ordering, $G'\prec G$
in this case as well. \end{proof}
\begin{lemma}
\label{lem:least option}Given a nonzero \textsc{Nim} position $\left(a_{1},a_{2},\ldots,a_{N}\right)$,
its $\prec$-least options are isomorphic to $\left(a_{1},\ldots,a_{N-1}\right)$.\end{lemma}
\begin{proof}
Since the quasi-lexicographic order prioritizes length, the $\prec$-least options are the ones where an entire heap is removed. Suppose that $a_i$ is removed, leaving the sequence $\left(b_{1},\ldots,b_{N-1}\right)$. If $a_{i}=a_{N}$, then this is isomorphic to $\left(a_{1},\ldots,a_{N-1}\right)$. Otherwise, let $j$ be the first index at which $\left(b_{1},\ldots,b_{N-1}\right)$ and $\left(a_{1},\ldots,a_{N-1}\right)$ differ, noting that $j\ge i$. Then we have $b_{j}=a_{j+1}>a_{j}$ so that $\left(a_{1},\ldots,a_{N-1}\right)\prec\left(b_{1},\ldots,b_{N-1}\right)$.
\end{proof}

\subsection{Finding Equivalence Classes}
\begin{lemma}
\label{lem:part ineq}Nonzero \textsc{Nim} positions are not equivalent to
any preceding position.\end{lemma}
\begin{proof}
We use induction. Let $G$ be a nonzero \textsc{Nim} position, with $G\cong\left(a_{1},\ldots,a_{N}\right)$. Assume that the claim is true for all positions preceding $G$. Note that by the cancellation property of Nim-sum (bitwise xor), and the strategy for mis\`{e}re \textsc{Nim}, $G$ can never have the same outcome as any of its options when added to arbitrary \textsc{Nim} positions (see Theorem V.1.1 of \cite{cgt})
, let alone all partizan games. In particular, $G$ is not equivalent to any positions of the same length beginning with $\left(a_{1},\ldots,a_{N-1}\right)$. 

Let $H$, a nonzero \textsc{Nim} position preceding $G$, be given. Denote the sequence corresponding to $H$ by $\left(b_{1},\ldots,b_{M}\right)$, and suppose $\left(b_{1},\ldots,b_{M-1}\right)\ncong\left(a_{1},\ldots,a_{N-1}\right)$. 
By the definition of the ordering, we must have $\left(b_{1},\ldots,b_{M-1}\right)\prec\left(a_{1},\ldots,a_{N-1}\right)$.
Then note that $\left(a_{1},\ldots,a_{N-1}\right)$ is the $\prec$-least option of $G$ (by Lemma \ref{lem:least option}) and no option of $G$ could be equivalent to $\left(b_{1},\ldots,b_{M-1}\right)$ by the induction hypothesis and Lemma \ref{lem:options prec}. 
Assume, for sake of contradiction that $G=H$. Then by condition (i) of Theorem \ref{thm:partizantest}, $G$ is equivalent to an option of $\left(b_{1},\ldots,b_{M-1}\right)$. But then, since equivalence is transitive, $H$ would be equivalent to an option of $\left(b_{1},\ldots,b_{M-1}\right)$. This is impossible by the induction hypothesis and Lemma \ref{lem:options prec}.

It remains to show that $G\neq 0$. But $G=0$ would imply $G\cong0$, by condition (iii) of theorem \ref{thm:partizantest}. 
\end{proof}
Since $\prec$ is a total order on the \textsc{Nim} positions (up to isomorphism),
we have:
\begin{theorem}
Equivalent \textsc{Nim} positions are isomorphic.
\end{theorem}

\section{Impartial Equivalence Classes}

\subsection{Paring Down the Nim Positions}

We will use the same well order on \textsc{Nim} positions as in Definition \ref{orderdef}.
However, in contrast to the partizan context, there are non-isomorphic positions that are known to be impartially equivalent. 
\begin{theorem}
\label{thm: adding 1}$n+1\equiv \left(n\oplus1\right)$, where $n\oplus1$ is a single heap of size $n+1$ if $n$ is even, and size $n-1$ if $n$ is odd.\end{theorem}
\begin{proof}
This is essentially the Mis\`{e}re Nim Rule from Ch. 13 of \cite{wwvol2}, which follows via a straightforward induction argument from the Mis\`{e}re Mex Rule (Theorem V.1.5 in \cite{cgt}).
\end{proof}
\begin{defn}
\label{def:red form}
The \textit{reduced form} of a given \textsc{Nim} position is obtained by
performing the following steps: \end{defn}
\begin{itemize}
\item[(i)] If there are at least two odd-sized heaps, replace the lowest pair
of them with heaps of size one less. Repeat this until there is at
most one odd-sized heap.
\item[(ii)] If there is an odd-sized heap, and it is not the largest heap, replace
the odd-sized heap with a heap of size one less and replace one of
the largest heaps with a heap of size one greater.
\item[(iii)] Delete all empty heaps.\end{itemize}
\begin{corollary}
\label{cor: red. equi.}The reduced form of a \textsc{Nim} position is impartially equivalent to the
original. \end{corollary}
\begin{proof}
The first two replacement rules preserve the impartial equivalence
class since by Theorem \ref{thm: adding 1}, we have $a+b\cong a+b+0\equiv a+b+1+1\cong\left(a+1\right)+\left(b+1\right)\equiv \left(a\oplus1\right)+\left(b\oplus1\right)$.
Also, note that $0\cong\left(\cdot\right)$ and $a+0\cong a$, so
that empty heaps can be discarded.\end{proof}
\begin{lemma}
\label{lem: red. prec.}The reduced form of a \textsc{Nim} position $G$ either
is isomorphic to $G$ or precedes $G$.\end{lemma}
\begin{proof}
The first replacement rule replaces a pair of odd heaps with heaps
of size one smaller. This is an option of an option of the original
position, so it would precede the original by Lemma \ref{lem:options prec}. 

The second replacement rule preserves the sum but moves one object
to the largest heap. If the only odd heap had size one, then this
decreases the length, otherwise it keeps the length the same but decreases
a non-maximum entry. In either case, the new position after the replacement precedes the position before the replacement. 

The last replacement rule does not affect the length of the sequences
we identify \textsc{Nim} positions with.\end{proof}
\begin{corollary}
\label{cor:red.opt.prec.}The reduced forms of all options of a \textsc{Nim} position $G$ precede $G$.\end{corollary}
\begin{proof}
This follows directly from Lemma \ref{lem:options prec} and Lemma \ref{lem: red. prec.}.\end{proof}
\begin{lemma}
\label{lem:least red.}Given a nonzero reduced-form \textsc{Nim} position $\left(a_{1},a_{2},\ldots,a_{N}\right)$,
the $\prec$-least reduced forms of the options are isomorphic to
$\left(a_{1},\ldots,a_{N-1}\right)$.\end{lemma}
\begin{proof}
First, we check that removing an entire heap leaves a position
in reduced form. If the original has no odd heaps, then neither would
an option with a heap removed. If the original has an odd heap, then
it must be the unique largest heap, $a_N$. In that case, removing a heap
either removes $a_{N}$ leaving no odd heaps or removes a smaller
heap so that the largest heap is still the only odd one.

By Lemma \ref{lem:least option}, $\left(a_{1},\ldots,a_{N-1}\right)$ precedes
all non-isomorphic options, so it remains to show that options which
do not involve removing an entire heap cannot have reduced forms preceding
$\left(a_{1},\ldots,a_{N-1}\right)$. Let $H$ be an option not obtained by removing an entire heap. Since $H$ has $N$ heaps, the only way $\left(a_{1},\ldots,a_{N-1}\right)$ would not precede the reduced form of $H$ is if the act of putting $H$ in reduced form decreases the number of heaps. The only way \emph{that} can happen is if $H$ contains a heap of size $1$.

If $N=1$, then $\left(a_{1},\ldots,a_{N-1}\right)\cong\left(\cdot\right)$,
which precedes all other options. If $N>1$, then since the original
position is in reduced form, we have $a_{N}>1$. If a heap of size
$a_{N}$ is replaced with $1$ by moving to $H$, the reduced form is $\left(a_{1},a_{2},\ldots,a_{N-2},a_{N-1}+1\right)$, which is preceded by $\left(a_{1},a_{2},\ldots,a_{N-2},a_{N-1}\right)$.
Otherwise, some other heap $a_i$ is replaced with $1$ by moving to $H$. If $a_{N}$ was even, then the reduced form is
$\left(a_{1},\ldots,a_{i-1},a_{i+1},\ldots,a_{N-1},a_{N}+1\right)$
which has length $N-1$ but is either preceded by or isomorphic to
$\left(a_{1},\ldots,a_{N-1}\right)$ since $a_{i+1}\ge a_{i}$. If
$a_{N}$ was odd, then the
reduced form is $\left(a_{1},\ldots,a_{i-1},a_{i+1},\ldots,a_{N-1},a_{N}-1\right)$.
Since $a_{N}-1\ge a_{N-1}$, this is preceded by or isomorphic to
$\left(a_{1},\ldots,a_{N-1}\right)$.
\end{proof}

\subsection{Finding Equivalence Classes}
\begin{lemma}
Nonzero reduced-form \textsc{Nim} positions are not impartially equivalent to any preceding
reduced-form position.\end{lemma}
\begin{proof}
We use induction. Let $G$ be a nonzero \textsc{Nim} position, with $G\cong\left(a_{1},\ldots,a_{N}\right)$. Assume that the claim is true for all reduced-form positions preceding $G$. By the same argument as in the proof of Lemma \ref{lem:part ineq}, $G$ is not impartially equivalent to any position beginning with $\left(a_{1},\ldots,a_{N-1}\right)$.

Let $H$, a nonzero \textsc{Nim} position in reduced form and preceding $G$, be given. Denote the sequence corresponding to $H$ by $\left(b_{1},\ldots,b_{M}\right)$, and suppose $\left(b_{1},\ldots,b_{M-1}\right)\ncong\left(a_{1},\ldots,a_{N-1}\right)$. 

By the definition of the ordering, we must have $\left(b_{1},\ldots,b_{M-1}\right)\prec\left(a_{1},\ldots,a_{N-1}\right)$
(and these are in reduced form by Lemma \ref{lem:least red.}). Then note
that $\left(a_{1},\ldots,a_{N-1}\right)$ is the $\prec$-least reduced-form
option of $G$ (by Lemma \ref{lem:least red.}) and no reduced form of an
option of $G$ could be impartially equivalent to $\left(b_{1},\ldots,b_{M-1}\right)$
by the induction hypothesis and Corollary \ref{cor:red.opt.prec.}. Assume, for sake of contradiction,
that $G\equiv H$. Then by condition (i) of Theorem \ref{thm:impartial test},
$G$ is impartially equivalent to an option of $\left(b_{1},\ldots,b_{M-1}\right)$.
But then, since impartial equivalence is transitive, $H$ would be
impartially equivalent to an option of $\left(b_{1},\ldots,b_{M-1}\right)$.
This is impossible by the induction hypothesis and Corollary \ref{cor:red.opt.prec.}.

It remains to show that $G\not\equiv 0$. Suppose, for sake of contradiction, that $G\equiv 0$. Condition (ii) of Theorem \ref{thm:impartial test} says that all of the options of the original position must have an option impartially equivalent to $0$. By the induction hypothesis, this can only happen if all of the options of the original position have $0$ itself as an option. As such, it must be that $N\le2$. The $N=1$ case is covered by the first paragraph, so it remains to check $N=2$ (as $G\ncong0$). For every option to have zero as an option, $a_{1}$ and $a_{2}$ must both be $1$ as otherwise there would be an option with two heaps, but $\left(1,1\right)$ is not reduced.
\end{proof}
Since $\prec$ is a total order on the \textsc{Nim} positions (up to isomorphism),
and every \textsc{Nim} position is impartially equivalent to a reduced-form
one by Corollary \ref{cor: red. equi.}, we have:
\begin{theorem}
Impartially equivalent \textsc{Nim} positions have isomorphic reduced forms.
\end{theorem}

\section{Transfinite Games}
In this section we consider transfinite non-loopy games: those which may have infinitely many distinct subpositions, but which still admit no infinite runs. In this setting, \textsc{Transfinite Nim} allows heap sizes to be arbitrary ordinals, although the number of heaps is still finite.

\begin{theorem}All theorems above apply in the (non-loopy) transfinite setting.\end{theorem}

\begin{proof}
Although chapter V of \cite{cgt} assumes all games are short, the theorems we cite in this paper do not require that assumption. Theorems 3.5 and 6.15 (and the propositions they rely on) do not use any induction. Theorems 3.6 and 6.16 use induction on $T$, but do not require $T$ to be short: transfinite induction suffices. In fact, as the proof of Theorem 3.6 was phrased in terms of ``minimal birthday'', it does not require any editing for the transfinite case. Also, Theorem 1.5, the Mis\`{e}re Mex Rule, still applies since the induction in the proof of Theorem 1.4 may as well be transfinite.

As mentioned in VIII.4 of \cite{cgt}, the strategy for \textsc{Nim} under normal play works identically in the transfinite case (using base-$2$ Cantor Normal Forms for the heap sizes). As the mis\`{e}re play strategy for \textsc{Nim} parallels the normal play strategy so well, and the number of heaps is finite (for counting heaps of size $1$ at the end of a game), the mis\`{e}re play strategy for \textsc{Nim} works identically in the transfinite case as well. In particular, the induction in the proof of the mis\`{e}re play strategy as Theorem V.1.1 of \cite{cgt} may as well be transfinite.

The induction required to get from the Mis\`{e}re Mex Rule to our Theorem \ref{thm: adding 1} can be transfinite. Inductions based on the number of heaps (such as the implicit one in \ref{def:red form}) do not require modification because the number of heaps is still finite in \textsc{Transfinite Nim}. All other inductions in this paper, such as in the proof of \ref{lem:part ineq}, do not require the heaps to be finite.
\end{proof}

\section*{Acknowledgments}
I thank Gregory Puleo for his gracious help with editing, and Douglas West, for ``The Grammar According to West.'' I also thank Meghan Allen, Paul Ottaway, and Thane Plambeck for helpful personal correspondence.


\begin{thebibliography}{1}

\bibitem[1]{allen} 
M.~R. Allen. Impartial combinatorial mis\`{e}re games, {\it miseregames.org} (2006), http://miseregames.org/\#related.

\bibitem[2]{wwvol2} 
E.~R. Berlekamp, J.~H. Conway, and R.~K. Guy, {\it Winning Ways for Your Mathematical Plays, Volume 2, Second Edition}, A K Peters, Natick, 2003.

\bibitem[3]{ONAG} 
J.~H. Conway, {\it On Numbers and Games, Second Edition}, A K Peters, Natick, 2001.

\bibitem[4]{logic} 
W. Li, {\it Mathematical Logic: Foundations for Information Science, Second Revised Edition}, Progr. Comput. Sci. Appl. Logic {\bf 25}, Birkh\"{a}user/Springer, Cham, 2010.

\bibitem[5]{zeroalone} 
G.~A. Mesdal and P. Ottaway. Simplification of partizan games in mis\`{e}re play, {\it Integers} {\bf 7(1)} (2007), G6.

\bibitem[6]{ends} 
R. Milley and G. Renault. Dead ends in mis\`{e}re play: the mis\`{e}re monoid of canonical numbers, {\it ArXiv} (2012), http://arxiv.org/abs/1212.6435v1.

\bibitem[7]{Plambeck} 
T.~E. Plambeck. Taming the wild in impartial combinatorial games, {\it Integers} {\bf 5(1)} (2005), G5.

\bibitem[8]{advances} 
T.~E. Plambeck. Advances in losing, {\it Games of No Chance 3}, Math. Sci. Res. Inst. Publ. {\bf 56} (2009), 57-89.

\bibitem[9]{siegelminor}  
A.~N. Siegel. Mis\`{e}re games and mis\`{e}re quotients, {\it ArXiv} (2006), http://arxiv.org/abs/math/0612616v2.

\bibitem[10]{siegel} 
A.~N. Siegel. Mis\`{e}re canonical Forms of partizan games, {\it ArXiv} (2007), http://arxiv.org/abs/math/0703565v1.

\bibitem[11]{cgt} 
A.~N. Siegel, {\it Combinatorial Game Theory}, Grad. Stud. Math., Amer. Math. Soc. {\bf 146}, Providence, 2013.
\end{thebibliography}
\end{document}